\documentclass[oneside,english]{amsart}
\usepackage[T1]{fontenc}
\usepackage[latin9]{inputenc}
\usepackage{amsthm}
\usepackage{amssymb}

\numberwithin{equation}{section} 
\numberwithin{figure}{section} 
\theoremstyle{plain}
\theoremstyle{plain}
\newtheorem{thm}{Theorem}
  \theoremstyle{plain}
  \newtheorem{prop}[thm]{Proposition}
  \theoremstyle{plain}
  
  \theoremstyle{remark}
  \newtheorem{rem}[thm]{Remark}
  \theoremstyle{plain}
 \newtheorem{definition}[thm]{Definition}

\usepackage{babel}

\begin{document}

\title{Geometric Spaces with No Points}

\author{Robert S. Lubarsky\\
 Florida Atlantic University\\
 Dept. of Mathematical Sciences \\
 777 Glades Rd., Boca Raton, FL 33431, USA\\
 Robert.Lubarsky@alum.mit.edu \\
 }

\maketitle
\thispagestyle{empty}
\begin{abstract}
Some models of set theory are given which contain sets that have
some of the important characteristics of being geometric, or
spatial, yet do not have any points, in various ways. What's
geometrical is that there are functions to these spaces defined on
the ambient spaces which act much like distance functions, and
they carry normable Riesz spaces which act like the Riesz spaces
of real-valued functions. The first example, already sketched in
\cite{FH}, has a family of sets, each one of which cannot be
empty, but not in a uniform manner, so that it is false that all
of them are inhabited. In the second, we define one fixed set
which does not have any points, while retaining all of these
geometrical properties.
\end{abstract}

\section{Introduction}

There has been increasing interest over the years in the
point-free approach to mathematics. Examples of this include
toposes as a model of set theory, in which the objects play the
role of sets and arrows functions, and the members of sets play no
role in the axiomatization (see \cite{MM} for a good introduction
and further references); locales (or frames) for the study of
topology, where the opens of a topological space are the objects
of the locale and the points of the space again are not accounted
for in the axiomatization \cite{J}; and point-free geometry, first
developed by Whitehead almost a century ago \cite{Wh1,Wh2,Wh3}.
This paper is a contribution to point-free geometry and analysis,
albeit from a more modern perspective than Whitehead's.

Part of the motivation for this approach is increased uniformity of
the results. A statement that applies to individual points obscures
the uniformity or continuity of the outcome. Point-free mathematics
also can provide stronger results and be more widely applicable, by
proving theorems with weaker hypotheses. Sometimes assumptions are
made in a theorem, such as Excluded Middle or Countable Choice, which
serve only to build certain points, whereas the actual content of
the construction lies elsewhere. By eliminating reference to points,
such powerful axioms can themselves be eliminated, and attention focused
on what's most important.

The framework of this paper is that of constructive mathematics,
and so Excluded Middle is not assumed. Since constructive
reasoning is a widely known and accepted paradigm, this is not
what is novel about this work. Rather, what is a bit different is
to work without Countable Choice. That is, we want to develop
models, necessarily violating CC, in which interesting phenomena
happen. This focus on models indicates a big difference from other
efforts. Much of the literature on point-free mathematics develops
theories and proves theorems in a point-free framework. The goal
of this work is to present some models.

\section{Polynomials without Uniform Roots}

In \cite{Ri}, a plausibility argument is given for why the
assertion that $X^{2}-a$ has a root ($a$ a complex number) is
tantamount to accepting a certain choice principle. The argument
comes down to the fact that there is no continuous square root
function on any neighborhood of 0 in the complex numbers. Of
course, if $a$ is in a simply connected region excluding 0, then
one of the two square roots can be chosen arbitrarily; if $a$
equals 0, then 0 itself is a square root; it is when we don't know
whether $a$ is 0 or not that we're not able to define a square
root. That choice plays a role here is indicated by the fact that
polynomials over $\mathbb{C}$ have roots in $\mathbb{C}$, if the
coefficients are limits of Cauchy sequences of complex rational
numbers \cite{R}; that every complex number (given as a pair of
Dedekind cuts in the reals) is the limit of a Cauchy sequence is a
typical application of Countable Choice. In \cite{FH}, this
plausibility argument is made more precise by casting it as a
topological model. The main point of this section is to provide
the details of this latter model, in which not all complex numbers
have square roots.

Before turning to this model, let's clarify its significance by
gathering some previously known results, working within IZF to be
definite. One would expect that the set of roots of any polynomial
over $\mathbb{C}$, as a simply defined subset of $\mathbb{C}$,
would enjoy some nice properties from having such a lineage. That
is indeed the case. For instance, the root set of a polynomial
over $\mathbb{C}$ is {\em quasi-located}, in the following sense:

\begin {definition}
For $S \subseteq \mathbb R$, the {\bf infimum} of $S$ is the real
number $\inf(S)$ such that $\inf(S)$ is a lower bound of $S$ --
i.e. for all $x \in S \; \inf(S) \leq x$ -- which is within any
$\epsilon > 0$ of $S$ -- $\forall \epsilon > 0 \; \exists x \in S
\; x < \inf(S) + \epsilon$.
\end {definition}

That works fine as long as $S$ is inhabited (has a member). If
it's not, as in the case at hand, we need a different definition.

\begin {definition}
For $S \subseteq \mathbb R$, the {\bf greatest lower bound} of $S$
is the real number ${\rm glb}(S)$ such that ${\rm glb}(S)$ is a
lower bound of $S$ -- i.e. for all $x \in S \; {\rm glb}(S) \leq
x$ -- which is bigger than any other lower bound -- for all $x$ if
$x$ is a lower bound of $S$ then $x \leq {\rm glb}(S)$.
\end {definition}

Corresponding to these two notion $\inf$ and ${\rm glb}$ are two
different notions of distance.

\begin {definition}
For $L$ a subset of a metric space, the {\bf distance} $d(x,L)$
from a point $x$ in the space to $L$ is $\inf_{y \in L}d(x,y)$,
where $d$ is the metric in the space. $L$ is {\bf located} if the
distance $d(x,L)$ exists for each $x$.
\end {definition}

For $L$ possibly not inhabited, we have the corresponding notion
using the ${\rm glb}$:

\begin {definition} \cite {LR}
For $L$ a subset of a metric space, the {\bf quasi-distance}
$\delta(x,L)$ from a point $x$ in the space to $L$ is ${\rm
glb}_{y \in L}d(x,y)$, where $d$ is the metric in the space. $L$
is {\bf quasi-located} if the quasi-distance $\delta(x,L)$ exists
for each $x$.
\end {definition}

For $L$ the root set of a polynomial, $L$ may not be inhabited,
but it is quasi-located \cite {LR}. Furthermore, it is shown in
the same paper that, for $D$ a closed disc containing a
quasi-located set $S$, the set of uniformly continuous functions
on $S$ that extend to uniformly continuous functions on $D$ is a
Riesz space $V$; the reader is reminded of the various Stone-like
representation theorems, by which every Riesz space is isomorphic
to a space of functions on a (non-empty) set, under certain
assumptions, such as Excluded Middle or Dependent Choice
\cite{F,CS}. So while the root set of some polynomial might not be
inhabited (i.e. have an element), it can be dealt with like an
inhabited set in many ways, via its quasi-distance function and
its Riesz space, giving it some metric and topological structure.
Finally, while the root set might not be inhabited, it is also not
empty \cite{LR}.

While we're on the subject of Riesz spaces, we'd like to address
one of their less discussed properties. In \cite{CS}, a
constructive Stone-Yosida representation theorem is proven: under
DC, if $R$ is a separable and normable Riesz space, then $R$ is
(isomorphic to) a set of continuous functions on a (Heine-Borel)
compact (i.e. complete and totally bounded) metric space. $R$ is
{\em normable} if it contains a strong unit 1, inducing an
embedding of $\mathbb{Q}$ into $R$, and, for each $f\in R,\;
U(f):=\{q\in\mathbb{Q}|q>f\}$ is a real number (i.e. a located
upper cut in the rationals), in which case $U$ provides a
pseudo-norm on $R$, which is a norm if $R$ is Archimedean. It is
not necessary for $R$ to be normable in order to be a function
space: just take $R$ to be the set of functions on \{0\} $\cup$
\{1 $\mid$ P\}, where P is some proposition; then the normability
of $R$ is equivalent to the decidability of P. But for $R$ to be a
nice function space on a nice set, some such hypothesis is
necessary. For instance:
\begin{prop}
If $R$ is a set of uniformly continuous functions on a compact
metric space, then $R$ is normable.
\end{prop}
The very simple proof is left as an exercise, and also follows easily
from results in \cite{B} or \cite{BV}. The purpose of this discussion
is to justify the assumption of normability in \cite{CS}, and the
effort spent proving normability in \cite{LR} and in the next section
here.

At this point, we are ready for the construction. Take the
standard topological model over $\mathbb{C}$. For those unfamiliar
with topological models, this can be viewed as the standard sheaf
model over $\mathbb{C}$, or as a Heyting-valued model over the
Heyting algebra of the open subsets of $\mathbb{C}$. As is
standard, $\mathbb{C}\Vdash$ IZF \cite{FH,G,L,MM,TvD}.
\begin{thm}
(Fourman-Hyland \cite{FH}) $\mathbb{C}\Vdash ``$Not every
polynomial has a root.$"$
\end{thm}
\begin{proof}
Let $G$ be the generic complex number. $G$ is characterized by the
relation $O\Vdash G\in O$ for any open set $O$. We claim that no
neighborhood of 0 forces that anything is a root of $X^{2}-G.$
Suppose to the contrary that $0\in O\Vdash z^{2}=G.$ $O$ contains
all circles with center 0 of sufficiently small radius. Consider
the circle in $O$ centered at 0 of radius $\epsilon^{2}$. Each
point $w$ of the circle is contained in a small open set $O_{w}$
forcing $z$ to be in an open set $U_{w}$ of diameter less than
$\epsilon$. $U_{w}$ must contain a square root of $w$, and each
square root of $w$ has absolute value $\epsilon$, so they are a
distance of $2\epsilon$ apart. Hence $U_{w}$ contains exactly one
such square root. Furthermore, for any other such neighborhood
$O_{w}'$ of $w$ (i.e. one determining $z$ to be in some $U_{w}'$
of diameter less than $\epsilon$), $U_{w}$ and $U_{w}'$ contain
the same square root of $w$. (If not, then $\emptyset\neq
O_{w}\cap O_{w}'\Vdash z\in U_{w}\cap U_{w}'=\emptyset.$) Notice
that the square root of $w$ so determined is a continuous function
of $w$: by choosing $O_{w}$ to be sufficiently small, $U_{w}$ can
be made arbitrarily small, and the values of the square root
function on the elements of $O_{w}$ can be limited to an
arbitrarily small arc. This, however, is a contradiction, as there
is no continuous single-valued square root function on a circle
around the origin.

This argument shows that $\mathbb{C}\not\Vdash``X^{2}-G$ has a
root.$"$ In contrast, it is not hard to see that
$\mathbb{C}-\{0\}\Vdash``X^{2}-G$ has a root.$"$ However, given
any open set, it can be translated to contain the origin, to come
up with a similar example. That is, for any $v\in\mathbb{C}$,
arguments similar to the ones above will show that no neighborhood
$v$ can force {}``$X^{2} - (G-v)$ has a root.$"$ So no non-empty
open set forces {}``every polynomial has a root.$"$ So
$\mathbb{C}\Vdash$ {}``Not every polynomial has a root.$"$
\end{proof}

\section{A Set of Complex Numbers with No Members}

In the previous section, we identified a family of subsets of
$\mathbb{C}$ which are all nonempty but not all inhabited. In this
section, we present a subset of $\mathbb{C}$ which is not
inhabited but still has a nontrivial distance function and
nontrivial Riesz space.

Let the topological space $F$ consist of all finite, non-empty
\footnote {If we allowed the empty set $\emptyset$ as a member of
$F$, under the given topology $\{\emptyset\}$ would be clopen, and
the model over $F$ would then split into two separate models: the
part over $\{\emptyset\}$, which is just $V$, and the part we're
interested in.} subsets of $\mathbb{C}$ . (Since we are dealing
only with the topological and metric properties here, you can just
as well think of taking $\mathbb{R}{}^{2}$ instead.) The topology
is that induced by the Hausdorff metric, which is also known as
the Vietoris topology. We give a self-contained description.
Intuitively, a basic open set $U$ is determined by finitely many
pieces of information. Information is either positive or negative.
A positive piece of information is an open set $O$ of
$\mathbb{C}$, and holds of $A\in F$ iff $A$ contains a point of
$O$. A negative piece of information is an open set $N$ of
$\mathbb{C}$, and holds of $A\in F$ if and only if $A \subseteq
N$. Notice that the finitely many pieces of negative information
can be combined into one piece, by taking the intersection of the
determining open sets, so we will assume that any basic open $U$
has only one piece of negative information. A point $A\in F$ is in
an open set $U$ if and only if $A$ satisfies all of the
information determining $U$. This can be summarized in the
following:

\begin {definition}
A basic open set $U_{\{ O_i \}_{i \in I}, N}$ is given by a
collection of open sets $O_i$ indexed over a finite set $I$ and an
open set $N$. $A \in U_{\{ O_i \}_{i \in I}, N}$ iff for each $i
\in I \; A \cap O_i \not = \emptyset$ and $A \subseteq N$.
\end {definition}

Since the information determining $U$ can be of any finite size,
such sets $U$ are a basis, and so determine a topology on $F$. The
model desired is the full topological model built on $F$. (As
above, for background on topological models, see any of
\cite{FH,G,L,MM,TvD}.)

An open set $U$ is in \emph{normal form} if its determining open
set $N$ is the union of the determining open sets $O_i$. Observe
that the open sets in normal form are a basis for the topology on
$F$.

In general, just like with classical forcing, the standard
topological model over a space $T$ introduces a generic $G$,
determined by the relation $O \Vdash G \in O$. In this model, $T$
is $F$ the finite subsets of $\mathbb C$, and so the canonical
generic, call it $H$, could be viewed as a subset of $\mathbb{C}$,
defined via the relation $U \Vdash H\subseteq\bigcup_i O_i$, where
$U$ is in normal form. (More generally, $U \Vdash H \subseteq N$,
where $N$ is the negative information in $U$, $U$ any basic open
set.) $H$ is indeed the set we want.
\begin{thm}
$F \Vdash ``H$ does not contain any points.$"$\end{thm}
\begin{proof}
Suppose to the contrary $U\Vdash X\in H$. Let $A\in U$. Say that
$A$ has $n$ elements. Shrink $U$ to an open $V = V_{\{ O_i
\}_{i<n}, N}$ containing $A$ such that the $n$-many opens $\{ O_i
\}_{i<n}$, each necessarily containing exactly one point of $A$,
are at least a distance $\delta>0$ apart from each other \footnote
{This brings up a subtle point in the meta-theory. Classically
there would be no trouble covering $A$ with open sets, one per
member of $A$. Constructively you might not be sure, as equality
on $\mathbb C$ is not assumed to be decidable. However, we defined
$F$ to consist of finite sets. A finite set is one which is
bijectible with a natural number. If we're considering a finite
set $A = \{a_0, a_1, ... , a_{n-1}\}$, then all of those numbers
are unequal to one another. The more general concept, which allows
for some members to be equal, is that of {\em finite
enumerability}.}. Shrink $V$ to $W \ni A$ in normal form forcing
$X$ to be within $\delta/2$ of some point with rational
coordinates. Note that $W$ has the form $W_{\{ O_{i}^{'} \}_{i<n},
N^{'}}$, that is, $W$ is determined by $n$-many open sets, and
that $W$ forces $X$ to be in a fixed one of the $O'_i$'s, say
$O'_j$. Let $z_0$ be the unique point in $A \cap O'_j$.

There is an $\iota > 0$ so small that the closed disc of radius
$\iota$ having $z_0$ as its right-most point is contained within
$O'_j$. As the right-most point, $z_0$ has angular measure 0, when
measuring angles the standard way. Now choose $B_0 \in W$ to agree
with $A$ except that from $O'_j$ $B_0$ has in addition to $z_0$
also the point diametrically opposite $z_0$ in this disc, $z_\pi$.
For any angle $\theta$ between 0 and $\pi$ inclusively let
$B_\theta$ be just like $B_0$ except the points $z_0, z_\pi$ have
been rotated around the disc by angle $\theta$ to $z_\theta,
z_{\theta + \pi}$. Around each $B_\theta$ is an open set forcing
$X$ to be closer to (say within some $\epsilon > 0$, small
relative to $\iota$, of) either $z_\theta$ or $z_{\theta + \pi}$.
Moreover, the choice between $\theta$ and $\theta + \pi$ is
unique, as any two such open sets are compatible, both containing
$B_\theta$. Let $f$ be the function so determined: the domain is
$[0, \pi]$, and some neighborhood of $B_\theta$ forces $X$ to be
closer to $z_{f(\theta)}$ than to $z_{f(\theta) + \pi}$. Without
loss of generality, say $f(0) = 0$. Notice that $f$ is continuous:
if $B_\theta \in W_\theta \Vdash ``X$ is closer to $z_\theta$
(resp. $z_{\theta + \pi})"$, then for $\phi$ sufficiently close to
$\theta$ $B_\phi$ is also in $W_\theta$, which would then force
$X$ to be closer to $z_\phi$ (resp. $z_{\phi + \pi}$). Hence by
continuity and $f(0) = 0$, $f$ is the identity function, and so
$f(\pi) = \pi$. But this can't happen, because $B_\pi = B_0$, so
$f(\pi)$ must equal $f(0)$ which is 0.

By this contradiction, no neighborhood $U$ can force any $X$ to be
in $H$.
\end{proof}
In a very real sense, $H$ is just the empty set. But it's really
not. Consider distance. Let $U$ be in normal form, with each
determining positive open set of diameter less than $\epsilon$.
Intuitively, the distance from $z\in\mathbb{C}$ to $H$ is between
$\rho={\rm dist}(z,\bigcup \{O_i\}),\,\{O_i\}$ the positive
information in $U$, and $\rho+\epsilon$. We could just give this
as the definition of distance and be done with it. Instead, we
would rather put this in a broader context. For that, we will
first need to develop the Riesz space aspect, and then return to
the discussion of distance.

To phrase it imprecisely but hopefully suggestively, the Riesz
space desired is the one generated by what would be the
projections of $H$ onto the $x$ (real) and $y$ (imaginary) axes,
if $H$ did have any points. To phrase it precisely, working in the
ambient (i.e. external) model, let $R$ be the Riesz space of
functions from $\mathbb{C}$ to $\mathbb{R}$ generated by the three
functions $1$ (the constant function with value $1$), $x$
(projection onto the real part), and $y$ (projection onto the
imaginary part). Each element of $R$ can be thought of as a word
in the language of Riesz spaces, an expression built from the
generators using vector addition, scalar multiplication, and the
lattice operations; or better: an equivalence class of such words,
mod functional equality. The internal Riesz space in the
topological model, ambiguously also called $R$, consists of those
same functions, or words. Given any $r,s\in R,$ $U\Vdash r=s$ iff,
for $N$ the negative information determining $U$ and $z\in N,\:
r(z)=s(z)$. The Riesz space structure is inherited from the
external $R$, and the equality axioms are clearly satisfied. There
remains only one fact to check.
\begin{prop}
$F\Vdash R$ is normable.\end{prop}
\begin{proof}
Let $r\in R$ and $A\in F$. Notice that $r$ as a function
externally is continuous. Cover $A$ with disjoint open sets $O_i$
so that $r$ varies less than $\epsilon$ on each of them. Let $U$
be the open set of $F$ in normal form the positive information of
which is given by $\{O_i\}$. Then $U$ forces $\sup(r)$ to be
between $\sup_{z\in O_i, i \in I}\, r(z)$ on the upper end and
$\sup_i\, \inf_{z\in O_i}\, r(z)$ on the lower, a distance less
than $\epsilon$.
\end{proof}
In \cite{CS}, Coquand and Spitters asked whether it is possible to
construct a Riesz space homomorphism into $\mathbb{R}$ of a
discrete, countable Riesz space, where the Riesz space is a vector
space over $\mathbb{Q}$, or of a separable Riesz space, without
any Choice axiom. The example above shows this is not always
possible, except for the issue of $R$ being discrete, which it is
not: if $z$ is on the boundary of $\{z\,|\, r(z)=s(z)\}$ then no
neighborhood of any $A$ containing $z$ will decide whether $r=s$.

\begin {prop}
It is consistent with IZF that there is a countable Riesz space
over $\mathbb Q$ with no Riesz space homomorphism into $\mathbb
R$, and a separable Riesz space over $\mathbb R$ with no Riesz
space homomorphism into $\mathbb R$.
\end {prop}

\begin {proof}
In the example above, let $U_{\{ O_i \}_{i \in I}, N}$ be an open
set in normal form in which each $O_i$ is a rectangle with
rational sides, say $O_i = (x^i_{\min}, x^i_{\max}) \times
(y^i_{\min}, y^i_{\max})$. Let $r_i$ be the Riesz space element
$(x - x^i_{\min}) \wedge (x^i_{\max} - x) \wedge (y - y^i_{\min})
\wedge (y^i_{\max} - y)$. On $O_i \; r_i$ as a function is always
positive. So on $N \; \bigvee_i r_i > 0$.

If there were such a homomorphism $\sigma$ then $U_{\{ O_i \}_{i
\in I}, N} \Vdash ``\sigma(\bigvee_i r_i) > 0."$ Since a Riesz
space homomorphism must satisfy $\sigma(\bigvee_i r_i) = \bigvee_i
(\sigma(r_i))$, $U_{\{ O_i \}_{i \in I}, N}$ can be covered by
open sets, each one forcing $(\sigma(x), \sigma(y))$ into one of
the $O_i$'s. Hence $(\sigma(x), \sigma(y))$ is forced to be a
member of $H$. But we have seen that $H$ has no members. Hence
there is no such $\sigma$.

In the description of $R$ given above as being determined by
generators, the field over which it was generated was not
mentioned. If the field is taken to be the rationals, then $R$ is
countable, if the reals, then $R$ is separable, in both cases
because it is so externally.
\end {proof}

Using $R$, we can now deal with distance.

\begin {prop}
In the example above, there is a non-trivial distance function to
$H$. \end {prop}

\begin {proof}
There is first the choice of which metric to use. While the
Euclidean metric is the most common, it is easier to do the
taxicab ($L^{1}$) metric, which we will do first. For a warm-up,
let's consider what the distance $d(0,H)$ from the origin $0$ to
the set $H$ should be. Classically, the distance would be
calculated between $0$ and each point $h \in H$
--- namely, $|x|+|y|$, where $x$ and $y$ are $h$'s coordinates ---
and the minimum would be taken over all such values. Observe that
this is exactly what the Riesz space is set up to do:
$d(0,H)=\inf(|x|+|y|)$. More generally, for $z = (x_z, y_z)$,
$d(z,H)=\inf(|x-x_z|+|y-y_z|)$.

To define Euclidean distance, we need to have squares available.
This calls for an expanded Riesz space. When generating $R$, close
not just under the Riesz space operations, but also squaring.
Everything else remains the same. Given that squares are
available, Euclidean distance can be defined as
$d(z,H)=\sqrt{\inf(|x-x_{z}|^{2}+|y-y_{z}|^{2}})$; notice that
there is no problem taking the square root, since it is the
non-negative root of a non-negative real number, and so always
exists.
\end {proof}

\begin{rem}
An interesting question is based on the observation that the
finiteness of the members $A\in F$ is not essential. The proof
given above is unchanged if we allow $A$ to be an arbitrary
compact subset of $\mathbb{C}$ (making allowances for the fact
that no open set is covered by opens of the form $U_{\{ O_i \}_{i
\in I}, N}$ in which the $O_i$'s are mutually disjoint -- in
general, overlaps of the open sets used as positive information
must be allowed). Then we're dealing with a different topological
space $E$, which is the completion of $F$. How do the models built
over $F$ and $E$ differ? Are they elementarily equivalent?
\end{rem}
$\,$
\begin{rem}
This construction brings up questions regarding dimension.
Distance was defined as in a two-dimensional space, and we knew to
do that because the original construction was based on a 2-D
space. How could the dimensionality of the missing underlying
space be recovered from the Riesz space alone? More generally,
under what circumstances can which properties of the invisible
underlying space be inferred from just the Riesz space?
\end {rem}
$\,$
\begin {rem}
This construction can be viewed as producing a set, the points of
which can have their $x$-coordinates determined, and their
$y$-coordinates too, just not both together. Is it possible to
build a 3-D model with no points in which you can determine any
two coordinates from among $x,y,$ and $z$ simultaneously, just not
all three?\end{rem}

\section {Addendum}

After the publication of the preceding, as a stand-alone article
\cite{L0}, I received some penetrating and well justified
questions from a colleague, making it clear that what is folklore
to some can be unclear to others. This note is to clarify some of
the background assumed in that first article. We will also take
this opportunity to correct some misleading and one mistaken
assertion in the same. (The mistake is in a comment; all the
results are correct.)

Topological models are just like the better known Boolean-valued
or forcing models, before modding out by ``not not.'' Hence it
might be useful to describe the situation in the context of
classical forcing first.

A set in any forcing extension is given by a term in the forcing
language of the ground model. A term is any set of the form $\{
\langle p_i, \sigma_i \rangle \mid i \in I\}$, where $p_i$ is a
forcing condition (a member of the forcing partial order
$\mathcal{P}$), $\sigma_i$ a term (inductively), and $I$ any index
set. The generic object $G$ is given by $\{ \langle p, \hat{p}
\rangle \mid p \in \mathcal{P}\}$, where $\hat{}$ is the embedding
of the ground model into the terms: inductively, $\hat{x} = \{
\langle \top, \hat{y} \rangle \mid y \in x\}$. $G$ is
characterized by the relation $p \Vdash ``\hat{p} \in G"$.

Often $G$ is identified with some other set easily definable from
$G$. Perhaps it would be easiest to illustrated what's going on
here via an example. Consider the simplest forcing partial order,
Cohen forcing, with conditions $2^{<\omega}$, finite binary
sequences. In a generic extension, $G$ is a set of finite binary
sequences. Typically, though, people work instead with $\bigcup
G$, which is an infinite binary sequence. To a set theorist,
$\bigcup G$ is easily definable from $G$, and in the other
direction $G$ is the set of proper initial segments of $\bigcup
G$. It's usually more convenient to work with $\bigcup G$, and so
set theorists take advantage of this simple inter-definability,
and, by abuse of language, refer to $\bigcup G$ as the generic
$G$.

To make the analogy with topological models tighter, the partial
order $2^{<\omega}$ is a basis for the Cantor space $2^\omega$.
The open set $O(p)$ corresponding to $p$ is the set of sequences
with $p$ as an initial segment. $p$ could be viewed as a name for
$O(p)$. Using open sets instead of the partial order, $G$ is then
defined as $\{ \langle O, \hat{O} \rangle \mid O$ is an open set
of $2^\omega\}$, and is characterized by $O \Vdash ``\hat{O} \in
G"$, not by $O \Vdash ``G \in \hat{O}"$, much less, as stated in
the original paper (p. 5), by $O \Vdash ``G \in O",$ which is not
merely mistaken, but actually incoherent, as $O$ is not a term in
the language.

Even if the latter is the way it's usually best thought of.

That needs explanation. Still working classically with
$2^{<\omega}$, note that $G$'s alter ego $\bigcup G$ is a member
of the topological space under consideration, $2^\omega,$ and is
characterized by $p \Vdash \bigcup G \in O(p),$ at least when
properly interpreted. $\bigcup G$ is certainly not in
$\hat{2^\omega}$, which is $2^\omega$ as interpreted in the ground
model $M$, also written as $(2^\omega)^M$, because $\bigcup G$ is
not even in $M$. Rather, $\bigcup G \in (2^\omega)^{M[G]}$.
Similarly, we could not say that $p \Vdash \bigcup G \in
\hat{O(p)};$ rather, $p \Vdash \bigcup G \in O(p),$ with the
latter $O(p)$ being the open set determined by $p$ in the
extension $M[G]$.

Now we're in a position to explain the mysterious, misleading, and
sometimes even mistaken comments of \cite{L0}. In Theorem 6 (p. 4)
of the original paper, where we're forcing (i.e. taking the
topological model) over the complex numbers $\mathbb{C}$, we refer
to the generic complex number as given by $O \Vdash ``G \in O".$
Strictly speaking, the generic $G$ is given by $O \Vdash ``\hat{O}
\in G".$ But the open sets of diameter less than $\epsilon$ cover
the space, so the generic determines a new complex number (as a
Dedekind cut). By simple inter-definability, this new number is
itself called $G$. It is determined by $O \Vdash ``G \in O",$
where the latter $O$ is viewed as the interpretation in the
extension of some ground-model description of $O$. For an example
of what such a description might be, every open set of complexes
is a union of countably many discs with rational center and
radius, and so can be described by a sequence of such
center-radius pairs $(\langle c_n, r_n \rangle)_{n \in \omega}$.

The story with the construction in section 3 of \cite {L0} is
similar, but more complicated. It is always the case that there is
a generic $G$ determined by $O \Vdash ``\hat{O} \in G".$ One
expects a $G'$, equidefinable with $G$, with $O$ forcing $G'$ to
be in $O$ as interpreted in the extension. While something along
those lines is likely to be true, it is not always so
straightforward as in the cases above. Section 3 is a case in
point. By analogy with simpler instances, one might reasonably
have guessed that the generic object $G$ is roughly the same as a
new member, there called $H$, of the topological space $F$, a
finite, non-empty set of complex numbers. The import of Theorem 8
is that that does not work, that the definition of $H$ as a set of
points leads to the empty set.

Of course, there is still $G$, and to understand any particular
topological model is to understand $G$. In some sense, $G$ is
(equidefinable with) a generalized member of the topological
space; the challenge is to determine what ``generalized'' means.
In this case, $G$ is the same as the distance function of
Proposition 11. (Since that distance function was derived from the
Riesz space $R$, $G$ can just as well be identified with $R$.) In
the earlier paper, the distance function of the model was given;
what remains to be done is to show that the generic can be
recovered from the distance function. This is not hard. Given $A
\in F,$ consider the open neighborhood $U$ of $A$ in normal form
with a positive piece of information $O_x$ for each $x \in A,$
with $O_x$ a disc with center $x$ and radius $\epsilon$ small
compared to the distances between the points in $A$. On all of the
points on the boundary of $O_x$, the distance function has a value
less than $2\epsilon.$ So some open set of $U$ is within the
circle of radius $2\epsilon$ around that boundary point. The
intersection of all those boundary points is exactly $O_x$. This
recovers $O_x$ from the distance function.

The situation is similar to the first construction of \cite {L4}.
There the topological space is the set of infinite, bounded
sequences of natural numbers. The generic is an infinite sequences
of naturals, but it's not bounded. Rather, it's pseudo-bounded, a
weakening of boundedness. Since classically boundedness and
pseudo-boundedness are equivalent, the generic can be viewed as a
new member of the space, as long as the space is understood as the
set of pseudo-bounded sequences. Similarly, in the case of
interest here, the generic cannot be viewed as a new member of the
space of finite, non-empty sets of complex numbers. Instead, in
the classical meta-theory, the finite, non-empty sets can be
identified with the distance functions on $\mathbb{C},$ and the
generic understood as a new distance function, interestingly with
no corresponding set of points from which it measures the
distance.

\end{document}